%% file: root.tex
\documentclass[10pt, a4paper, oneside, reqno]{amsart}

\input{style_PM}
\input{commands_PM}
\graphicspath{{images/}}

\usepackage[numbers,compress]{natbib}
\usepackage{hyperref}       
\usepackage{url}            

\usepackage{todonotes}
\usepackage[algo2e,ruled,vlined]{algorithm2e}

\SetKw{Return}{Return}
\SetKw{Find}{Find}
\SetKw{Define}{Define}

\title[Robust Control Design for Linear Systems via Multiplicative Noise]{Robust Control Design for Linear Systems via Multiplicative Noise}

\author[First]{Benjamin Gravell}
\author[Second]{Peyman Mohajerin Esfahani}
\author[Third]{Tyler Summers}

\thanks{The authors are with the Control, Optimization, and Networks lab, UT Dallas, The United States ({\tt\{benjamin.gravell, tyler.summers\}@utdallas.edu}), and the Delft Center for Systems and Control, TU Delft, The Netherlands ({\tt P.MohajerinEsfahani@tudelft.nl}). This material is based upon work supported by the United States Air Force Office of Scientific Research under award number FA2386-19-1-4073.}

\begin{document}

\maketitle

\begin{abstract}
Robust stability and stochastic stability have separately seen intense study in control theory for many decades.
In this work we establish relations between these properties for discrete-time systems and employ them for robust control design.
Specifically, we examine a multiplicative noise framework which models the inherent uncertainty and variation in the system dynamics which arise in model-based learning control methods such as adaptive control and reinforcement learning.
We provide results which guarantee robustness margins in terms of perturbations on the nominal dynamics as well as algorithms which generate maximally robust controllers.
\end{abstract}


\section{Introduction}
Model-based learning control, which encompasses classical system identification (e.g. \cite{ljung2001system}) and adaptive control (e.g. \cite{astrom2013adaptive}) as well as branches of modern reinforcement learning (e.g. \cite{nagabandi2018neural,tu2019gap}), universally uses a stochastic data model, where a model is estimated from data corrupted by random noise. 
A salient perennial issue in these methods is ensuring stability despite the presence of concomitant model errors; this is the problem of \emph{robustness}.

Traditional methods for designing robust controllers include $\mathcal{H}_\infty$ control design, which treats modeling error as a worst-case or adversarial  disturbance (\cite{zhou1998essentials,basar1998}), robust optimization over parametric state-space uncertainty sets, which typically involve 
searching for shared Lyapunov functions via convex semidefinite programming (\cite{khargonekar1990robust,nemirovskii1993several,boyd1994linear,Corless1994,de1999new}), and certainty-equivalent control, which utilizes only a nominal model and ignores modeling error entirely. However, since the robust design methods work with uncertainty sets, it is generally not straightforward to relate the uncertainty set descriptions to actual uncertainties arising from a stochastic data model.

Alternatively, in this paper we explore the connection between a special type of \emph{stochastic stability} and \emph{robust stability} and exploit this connection for robust control design.
In particular, we use a \emph{multiplicative noise} model where the noise is viewed as a representation of uncertainty in the nominal system model. This framework is naturally disposed toward trading off performance and robustness according to uncertainty directions and magnitudes which can be estimated from trajectory data during model-based learning control.
The study of multiplicative noise models has a long history in control theory (\cite{kushner1967stochastic,wonham1967optimal,kozin1969survey}). In contrast with the well-known additive noise setting, multiplicative noise captures linear dependence of the noise on the state and control input, which occurs intrinsically in a diverse array of modern control systems such as robotics (\cite{dutoit2011robot}), networked systems with noisy communication channels (\cite{hespanha2007survey}), modern power networks with high penetration of intermittent renewables (\cite{carrasco2006power}), turbulent fluid flow (\cite{lumley2007stochastic}). Linear systems with multiplicative noise are particularly attractive as a stochastic modeling framework because they remain simple enough to admit closed-form expressions for stability and optimal control via generalized Lyapunov and Riccati equations. A multiplicative noise model also holds a distinct advantage of being sensitive to \emph{structured} uncertainties in \emph{specific directions directly related to data}, as opposed to generic sets governed by norm balls as in \cite{Dean2019}.

In this paper we consider a \textbf{fundamental question}: \\
\emph{What is the set of perturbations to the system matrix where the perturbed system can be guaranteed stable, given knowledge only of the nominal system dynamics and stochastic stability of a system with multiplicative noise?} 

This question was considered by \cite{bernstein1987} for the continuous-time setting. Surprisingly, it was noted that the addition of multiplicative noise could actually \emph{stabilize} a deterministically unstable system when interpreted in the sense of Stratonovich (rather than It\^{o}) \cite{arnold1983stabilization}. Despite this subtle difficulty, combining mean-square stability of a multiplicative noise system with a \emph{right-shift} of the system dynamics, i.e., increasing the real parts of the eigenvalues of $A$ as $A \leftarrow A + c I$ was shown sufficient to ensure robust deterministic stability. Similarly, we develop conditions for discrete-time systems which combine mean-square stability of a multiplicative noise system with a \emph{scaling} of the system dynamics, i.e., increasing the absolute value of eigenvalues of $A$ as $A \leftarrow cA$.

In this paper we make the following contributions:
\begin{itemize}
    \item We develop a result utilizing shared Lyapunov functions that establishes robust stability of a set of deterministic systems given stochastic (mean-square) stability of another system with multiplicative noise (Theorem \ref{thm:robust_multi_shared_lyap}).
    \item We develop a complementary result utilizing an auxiliary system with scaled dynamics matrices that similarly establishes robust stability of a set of deterministic systems given stochastic (mean-square) stability of another system with multiplicative noise (Theorem \ref{thm:robust_multi_aux}).
    \item We show that both theorems yield robustness sets whose size increases monotonically with the multiplicative noise variances and collapse to zero in the case of zero noise.
    \item We develop a corresponding pair of algorithms which efficiently compute controllers that simultaneously maximize robustness and minimize a quadratic cost.
\end{itemize}

We elaborate on the robust stability problem in Section \ref{sec:problem_formulation}, develop theorems in Sections \ref{sec:robust_shared} and \ref{sec:robust_aux}, develop corresponding algorithms in Section \ref{sec:control_design}, give numerical examples in Section \ref{sec:numerical_results}, and conclude in Section \ref{sec:conclusion}.

\section{Problem formulation} \label{sec:problem_formulation}
Consider a discrete-time linear time-invariant (LTI) system
\begin{align} \label{eq:lti_true}
    x_{t+1} = \bar{A} x_t + \bar{B} u_t
\end{align}
where the entries of $\bar{A}$ and $\bar{B}$ are unknown constants and are approximated (perhaps from noisy trajectory data) by the known nominal matrices $A$ and $B$ leading to the nominal model
\begin{align} \label{eq:lti_nominal}
    x_{t+1} = A x_t + B u_t
\end{align}
where $x_t \in \RR^n$ is the system state, $u_t \in \RR^m$ is the control input, $A \in \RR^{n \times n}$ is the dynamics matrix and $B \in \RR^{n \times m}$ is the input matrix.
In order to stabilize the system in \eqref{eq:lti_nominal}, we use linear state feedback $u_t = K x_t$ with gain matrix $K \in \RR^{m \times n}$; classical results \cite{kalman1960contributions, kalman1963controllability} show that if the pair $(A,B)$ is controllable, then the closed-loop eigenvalues of $A+BK$ can be placed arbitrarily by choosing suitable gains.
A robust stabilization problem is to find a linear state-feedback control $u_t = K x_t$ such that the closed-loop nominal system remains stable under fixed perturbations of $A$ and $B$ i.e. that
\begin{align} \label{eq:lti_approx}
    x_{t+1} = \big( (A + \Delta A) + (B+\Delta B) K \big) x_t
\end{align}
is stable for some set of perturbations $\Delta A \in \mathcal{A}$ and $\Delta B \in \mathcal{B}$, ideally containing the true matrices $\bar{A}$ and $\bar{B}$.

As a parallel development, consider an LTI system with multiplicative noise with dynamics
\begin{align} \label{eq:ltim}
    \quad x_{t+1} = \Big(A  + \sum_{i=1}^p \gamma_{ti} A_i \Big) x_t +  \Big(B + \sum_{j=1}^q \delta_{tj} B_j \Big) u_t
\end{align}
Multiplicative noise terms are modeled by the i.i.d. across time (white), zero-mean, mutually independent scalar random variables $\gamma_{ti}$ and $\delta_{tj}$, which have variances $\alpha_i$ and $\beta_j$, respectively. The matrices $A_i \in \RR^{n \times n}$ and $B_i \in \RR^{n \times m}$ specify how each scalar noise term affects the dynamics and input matrices.

Stability of such a system depends on the behavior of the second moments (covariance) of the state over time as formalized by the notion of \emph{mean-square stability}, a form of robust stability which is stricter than stabilizability of the nominal system $(A,B)$ and limits the size of the multiplicative noise variances (\cite{kozin1969survey, willems1976feedback}):
\begin{Def}[Mean-square stability]
The system in \eqref{eq:ltim} is mean-square stable if and only if
\begin{align*}
    \lim_{t \rightarrow \infty} \EE \big[ x_t x_t^\tp \big] = 0 \ \forall \ \|x_0\| < \infty
\end{align*}
\end{Def}
In order to stabilize the system in \eqref{eq:lti_nominal}, we again use linear state feedback $u_t = K x_t$; mean-square stability of the closed-loop system with this control is equivalently characterized by the solution of a \emph{generalized Lyapunov equation} (GLE) (\cite{kleinman1969, boyd1994linear}):
\begin{Lem} \label{lem:mss_glyap}
The system in \eqref{eq:ltim} is mean-square stable in closed-loop with state feedback $u_t = K x_t$ if and only if for any $Q \succ 0$ there exists $P \succ 0$ satisfying 
\begin{align} \label{eq:glyap}
    P & =  Q  + (A+BK)^\tp P (A+BK) 
     + \sum_{i=1}^p \alpha_i A_i^\tp P A_i + \sum_{j=1}^q \beta_j K^\tp B_j^\tp P B_j K.
\end{align}
\end{Lem}
\begin{Cor}
    In the discrete-time setting, mean-square stability of \eqref{eq:ltim} with control $u_t = K x_t$ implies deterministic stability of \eqref{eq:lti_nominal} with the same control $u_t = K x_t$.
\end{Cor}
\begin{proof}
    From \eqref{eq:glyap}, strict mean-square stability implies existence of $P \succ 0$ such that
    \begin{align*}
        P   & = Q  + (A+BK)^\tp P (A+BK) 
            + \sum_{i=1}^p \alpha_i A_i^\tp P A_i + \sum_{j=1}^q \beta_j K^\tp B_j^\tp P B_j K \\
            & \succeq Q  + (A+BK)^\tp P (A+BK)
    \end{align*}
    which ensures stability of $A+BK$.
\end{proof}

One mean-square stabilizing control arises by solving the infinite-horizon multiplicative noise LQR problem
\begin{align}
    \underset{{\pi \in \Pi}}{\text{min}}  & \quad \EE_{\{\gamma_{ti}\}, \{\delta_{tj}\}} \sum_{t=0}^\infty \left(x_t^\tp Q x_t + u_t^\tp R u_t\right), \nonumber \\
    \text{s.t.}                           & \quad x_{t+1} = \Big(A  + \sum_{i=1}^p \gamma_{ti} A_i \Big) x_t +  \Big(B + \sum_{j=1}^q \delta_{tj} B_j \Big) u_t, \nonumber
\end{align}
where $Q\succeq 0$ and $R\succ 0$.
We assume that the problem data $A$, $B$, $\alpha_i$, $A_i$, $\beta_j$, and $B_j$ permit the existence of a finite solution, in which case the system is called \emph{mean-square stabilizable}. 
Dynamic programming can be used to show that the optimal policy is linear state feedback $u_t = K^* x_t$, where $K^* \in \RR^{m \times n}$ denotes the optimal gain matrix, and the resulting optimal cost $V(x_0)$ for a fixed initial state $x_0$ is quadratic, i.e., $V(x_0) = x_0^\tp P x_0$, where $P \in \RR^{n \times n}$ is a symmetric positive definite matrix. The optimal cost is given by the solution of the \emph{generalized algebraic Riccati equation} (GARE)
\begin{align*} 
P   & = Q + A^\tp P A + \sum_{i=1}^p \alpha_i A_i^\tp P A_i 
     -  A^\tp P B (R + B^\tp P B + \sum_{j=1}^q \beta_j B_j^\tp P B_j)^{-1} B^\tp P A \nonumber
\end{align*}
which can be derived similarly to the GARE given by \cite{mclane1972} for continuous-time systems.
The solution $P = \gare(A,B,Q,R,\alpha_i,\beta_j,A_i,B_j)$ can be obtained via the value iteration recursion
\begin{align*}
P_{t+1} & =  Q + A^\tp P_t A + \sum_{i=1}^p \alpha_i A_i^\tp P_t A_i 
        -  A^\tp P_t B (R + B^\tp P_t B + \sum_{j=1}^q \beta_j B_j^\tp P_t B_j)^{-1} B^\tp P_t A, \nonumber
\end{align*}
with $P_0 = Q$ or via semidefinite programming formulations (\cite{boyd1994linear,el1995state,li2005estimation}). The optimal gain is then
\begin{align*}
    K^* = -  \Big(R + B^\tp P B + \sum_{j=1}^q \beta_j B_j^\tp P B_j \Big)^{-1} B^\tp P A.
\end{align*}

\subsection{Generalized eigenvalues and semidefiniteness}
The following lemmas regarding generalized eigenvalue problems and semidefiniteness will be needed later:
\begin{Lem} \label{lem:gen_eig}
If $\lambda_{\max}$ is the maximum generalized eigenvalue which solves
$ Av = \lambda Bv $,
then
$\lambda_{\max} B \succeq A$.
\end{Lem}
\begin{Cor}
    If $B$ is singular, then $\lambda_{\max}$ becomes infinite.
\end{Cor}
We omit the proofs since these results are widely known; they follow readily from the method of Lagrange multipliers and Rayleigh quotients.

Every symmetric matrix $S$ can be split into positive and negative semidefinite parts via eigendecomposition as
\begin{align*}
    S = S^+ + S^-
\end{align*}
where
\begin{align*}
    S^+ = \sum_{i} \lambda_i v_i v_i^\tp \succeq S \succeq 0, \quad 
    S^- = \sum_{j} \lambda_j v_j v_j^\tp \preceq S \preceq 0, 
\end{align*}
where $\lambda_i$ and $\lambda_j$ are positive and negative eigenvalues respectively with associated eigenvectors $v_i$ and $v_j$.

\section{Robustness via shared Lyapunov functions} \label{sec:robust_shared}
We begin by ignoring the contribution of feedback control; we will introduce the control again in Sec. \ref{sec:control_design}.
We also restrict our search over $\Delta A$ to the set 
\begin{align*}
    \Delta A \in \mathcal{A} = \left\{ \sum_{i=1}^p \mu_i A_i \ \Big| \  \mu_i \in \RR,  0 \leq \mu_i < y \theta_i, \sum_{i=1}^p \theta_i = 1 \right\}
\end{align*}
The $\theta_i$ are scalars that represent the relative amount of uncertainty in each direction, while $y$ is a scalar governing the maximum magnitude of the perturbations. The $\theta_i$ and $y$ can be estimated from statistics of sampled trajectory data, e.g., using bootstrap resampling methods. This approach is intuitive; mean-square stability under stochastic instantaneous perturbations in specific directions $A_i$ ought to ensure deterministic stability under constant shifts of the dynamics in those same directions.
Note the number of linearly independent uncertainty directions $p$ is limited by the number of entries of $A$ i.e. $p \leq n^2$.
Consider the problem of finding the largest deviation scalar $y^*$ which can be tolerated while still guaranteeing stability of the perturbed deterministic system
\begin{align*}
    x_{t+1} = \big( A + \Delta A \big) x_t, \ \Delta A \in \mathcal{A}
\end{align*}
based on mean-square stability of the stochastic system
\begin{align*}
    x_{t+1} = (A+\gamma_{ti} A_i) x_t
\end{align*}
with $\EE [ \gamma_{ti} ] = 0$, $\EE [\gamma_{ti}^2 ] = \alpha_i > 0$.

\clearpage
\subsection{Scalar case}
First, we treat the scalar case where $n = p = 1$ so $A_1 = 1$ and $\theta_1 = 1$ without loss of generality.
\begin{Lem}
Suppose 
\begin{align*}
    x_{t+1} = (A + \gamma_{t}) x_t
\end{align*}
is mean-square stable where $A$, $x_t$, $\gamma_t $ are scalars with $\mathds{E} [ \gamma_{t}^2 ] = \alpha > 0$. 
Then, the perturbed deterministic system
\begin{align} \label{eq:perturbed_det_sys}
    x_{t+1} = (A + y) x_t
\end{align}
is stable for any fixed perturbation $|y| \leq \sqrt{A^2 + \alpha} - | A |$.
\end{Lem}

\begin{proof}
The GLE in \eqref{eq:glyap} reduces to
\begin{align*}
    P = Q + A^2 P + \alpha P
\end{align*}
where $P$, $Q$ are scalars with solution
\begin{align*}
    P = Q \left[ {1 - (A^2 + \alpha)} \right]^{-1}
\end{align*}
which is positive only when $\sqrt{A^2 + \alpha} < 1$. By assumption the system is mean-square stable, so Lemma \ref{lem:mss_glyap} implies that the solution $P > 0$
and thus indeed $\sqrt{A^2 + \alpha} < 1$.
By the restriction on $y$ and the triangle inequality 
\begin{align*}
    \rho(A + y) = |A + y | \leq |A| + | y | \leq \sqrt{A^2 + \alpha} < 1 ,
\end{align*}
proving stability of \eqref{eq:perturbed_det_sys}.
\end{proof}

This simple example demonstrates that the robustness margin increases monotonically with the multiplicative noise variance and when $\alpha = 0$, i.e. $|a| \rightarrow 1$, the bound collapses and no robustness is guaranteed.

\subsection{Multivariate case}

The optimal bound $y^*$ is found by solving the program
\begin{align} \label{eq:gen_SDP}
    \begin{aligned}
        & \underset{{y,P}}{\text{maximize}} && y \\
        &\text{subject to} 
        &&   P \succeq  I + A^\tp P A + \sum_{i=1}^p \alpha_i A_i^\tp P A_i \\
        &&&  P \succeq \left( A + \Delta A \right)^\tp P \left( A + \Delta A \right) \\
        &&& \Delta A = y \sum_{i=1}^{p} k_i \theta_i A_i \ \forall \ k_i \in \{-1,+1\}
    \end{aligned}
\end{align}
i.e. maximizing $y$ while ensuring that there exists a $P$ which generates a Lyapunov function which guarantees both mean-square stability of the stochastic system and deterministic stability of the perturbed deterministic system. Here we have arbitrarily chosen $Q = I$ e.g. as in \eqref{eq:glyap} without loss of generality since the constraints pertain only to stability, which is invariant to the choice of $Q$.
Since the program is quasiconvex in $y$, it can be solved by bisection over $y$ and solving a feasibility SDP for each fixed $y$, with the solution being the largest $y$ which admits a feasible solution to the SDP.

The set of constraints in the second line of \eqref{eq:gen_SDP} form corners of a convex box polytope in the space of $n \times n$ matrices, which is necessary and sufficient to guarantee stability of \eqref{eq:lti_approx} (\cite{boyd1994linear,Corless1994}). Thus, from the perspective of verifying stability of $A + \Delta A$ this procedure no better than simply solving the same program \eqref{eq:gen_SDP} with the first constraint deleted, which has a larger feasible set and thus will achieve at least as good a bound as \eqref{eq:gen_SDP}. However, the solution of \eqref{eq:gen_SDP} defines a hard upper limit on the following bounds we develop in this section which are based on a shared Lyapunov function, since \eqref{eq:gen_SDP} gives the optimal bound. The bounds we develop in this section trade optimality (conservativeness) for the assurance that $P$ guarantees stability of the perturbed deterministic system without explicitly using the Lyapunov inequality ${P \succeq (A+\Delta A)^\tp P (A+\Delta A)}$.

Giving up optimization over $P$ and instead choosing $Q$ arbitrarily (later in Sec. \ref{sec:control_design}, $Q$ will be chosen as the cost matrix of an LQR control design) and calculating the associated $P$, we obtain the following result:
\begin{Thm} \label{thm:robust_multi_shared_lyap}
Suppose
\begin{align} \label{eq:mult_system}
    x_{t+1} = \left( A+ \sum_{i=1}^{p} \gamma_{ti} A_i \right) x_t
\end{align}
is mean-square stable with $\EE [ \gamma_{ti} ] = 0$, $\EE [ \gamma_{ti}^2 ] = \alpha_i > 0$. \\
Fix a $Q \succeq I$ and the solution $P \succ 0$ to 
\begin{align} \label{eq:glyap_p}
    P = pQ + A^\tp P A + \sum_{i=1}^{p} \alpha_i A_i^\tp P A_i 
\end{align}
Let $\eta_i > 0$ be scalars which satisfy
\begin{align} \label{eq:NLMI}
    pQ + \sum_{i=1}^p \alpha_i A_i^\tp P A_i 
    & \succeq 
    \sum_{i=1}^p \eta_i \left( A_i^\tp P A + A^\tp P A_i \right)^+
     + \sum_{i=1}^p \sum_{j=1}^p \eta_i \eta_j \left( A_i^\tp P A_j + A_j^\tp P A_i \right)^+ 
\end{align}
Then the deterministic system 
\begin{align} \label{eq:det_system}
    x_{t+1} = \left( A + \sum_{i=1}^{p} \mu_i A_i \right) x_t
\end{align}
is deterministically stable for any 
\begin{align} \label{eq:robust_multi_shared_lyap_bound}
    0 \leq \mu_i < \eta_i
\end{align}
\end{Thm}

\begin{proof}
It is evident that valid $\eta_i > 0$ exist since ${pQ  + \sum_{i=1}^{p} \alpha_i A_i^\tp P A_i }$ is strictly positive definite.
Rearranging \eqref{eq:glyap_p} to ${pQ  + \sum_{i=1}^{p} \alpha_i A_i^\tp P A_i = P - A^\tp P A }$ and substituting gives
\begin{align*}
    P
    & \succeq A^\tp P A + \sum_{i=1}^p \eta_i \left( A_i^\tp P A + A^\tp P A_i \right)^+ \nonumber 
    + \sum_{i=1}^p \sum_{j=1}^p \eta_i \eta_j \left( A_i^\tp P A_j + A_j^\tp P A_i \right)^+ \\
    & \succeq A^\tp P A + \sum_{i=1}^p \mu_i \left( A_i^\tp P A + A^\tp P A_i \right)^+ \nonumber  
    + \sum_{i=1}^p \sum_{j=1}^p \mu_i \mu_j \left( A_i^\tp P A_j + A_j^\tp P A_i \right)^+ \\
    & \succeq A^\tp P A + \sum_{i=1}^p \mu_i \left( A_i^\tp P A + A^\tp P A_i \right) \nonumber  
    + \sum_{i=1}^p \sum_{j=1}^p \mu_i \mu_j \left( A_i^\tp P A_j + A_j^\tp P A_i \right) \\
    & = \left( A + \sum_{i=1}^{p} \mu_i A_i \right)^\tp P \left( A + \sum_{i=1}^{p} \mu_i A_i \right)
\end{align*}
which proves stability of \eqref{eq:det_system}.
\end{proof}

\begin{Rem}
    The unidirectional bound in \eqref{eq:robust_multi_shared_lyap_bound} of Thm. \ref{thm:robust_multi_shared_lyap} can be made bidirectional by replacing \eqref{eq:NLMI} with
    \begin{align*}
    pQ + \sum_{i=1}^p \alpha_i A_i^\tp P A_i 
    & \succeq 
    \sum_{i=1}^p \eta_i Y_i 
    + \sum_{i=1}^p \sum_{j=1}^p \eta_i \eta_j Z_{ij} \nonumber        
    \end{align*}
    where
    \begin{align}
    \begin{array}{cc}
          Y_i     \succeq \left( A_i^\tp P A + A^\tp P A_i \right)^+ , \quad \text{ and } 
        & Y_i     \succeq -\left( A_i^\tp P A + A^\tp P A_i \right)^-, \\
          Z_{ij}  \succeq \left( A_i^\tp P A_j + A_j^\tp P A_i \right)^+, \text{ and } 
        & Z_{ij}  \succeq -\left( A_i^\tp P A_j + A_j^\tp P A_i \right)^-, 
    \end{array}
    \end{align}
    yielding the bidirectional bound $| \mu_i | < \eta_i$.
\end{Rem}

\clearpage
\begin{Rem}
    Let $\theta_i \geq 0$ be scalars such that $\sum_{i=1}^p \theta_i = 1$; these denote relative uncertainty in directions $A_i$.
    The largest robust stability bounds with respect to this choice of $\theta_i$ are obtained by setting $\eta_i = y \theta_i$ and maximizing the scalar $y$, which can be accomplished via bisection.
    As discussed earlier, optimizing a bidirectional bound over $P$, $Q$, and $y$ is equivalent to solving the full program in \eqref{eq:gen_SDP}.
\end{Rem}

For $p=1$, the Theorem \ref{thm:robust_multi_shared_lyap} reduces as follows:
\begin{Cor} \label{cor:robust_multi_shared_lyap_single}
Suppose
\begin{align*} 
    x_{t+1} = \left( A+ \gamma_{t1} A_1 \right) x_t
\end{align*}
is mean-square stable with $\EE [ \gamma_{t1} ] = 0$, $\EE [ \gamma_{t1}^2 ] = \alpha_1 > 0$. \\
Fix a $Q \succeq I$ and the solution $P \succ 0$ to 
\begin{align*}
    P = Q + A^\tp P A + \alpha_1 A_1^\tp P A_1
\end{align*}
Let $\zeta_1 > 0$ be a scalar which satisfies
\begin{align} \label{eq:NLMI_single}
    \frac{1}{\sqrt{\zeta_1^2 + \alpha_1} - \zeta_1} Q + 2 \zeta_1 A_1^\tp P A_1 \succeq (A^\tp P A_1 + A_1^\tp P A )^+ .
\end{align}
Then the deterministic system 
\begin{align*}
    x_{t+1} = \left( A + \sum_{i=1}^{p} \mu_i A_i \right) x_t
\end{align*}
is deterministically stable for any 
\begin{align*}
    0 \leq \mu_1 < \eta_1
\end{align*}
where $\eta_1 > 0$ is a scalar uniquely determined by $\zeta_1$ as
\begin{align*}
    \eta_1 = \sqrt{\zeta_1^2 + \alpha_1} - \zeta_1 \ \left( \ \leq \sqrt{\alpha_1} \ \right)
\end{align*}
Also, $\eta_1$ satisfies
\begin{align}
    Q + \alpha_1 A_1^\tp P A_1 
    & \succeq 
    \eta_1 \left( A_1^\tp P A + A^\tp P A_1 \right)^+ 
     + 2 \eta_1^2 A_1^\tp P A_1 \label{eq:NLMI_single2}
\end{align}
in accordance with Thm. \ref{thm:robust_multi_shared_lyap}.
\end{Cor}

\begin{proof}
Multiplying both sides of \eqref{eq:NLMI_single} by $\eta_1$ and using \\ 
${\eta_1 = \sqrt{\zeta_1^2 + \alpha_1} - \zeta_1}$ gives
\begin{align}
    Q + 2 \eta_1 \zeta_1 A_1^\tp P A_1 
    & = 
    \frac{\eta_1}{\sqrt{\zeta_1^2 + \alpha_1} - \zeta_1} Q + 2 \eta_1 \zeta_1 A_1^\tp P A_1 \nonumber \\
    & \succeq \eta_1 \left( A^\tp P A_1 + A_1^\tp P A \right)^+ \label{eq:cor_intermediate_lmi}
\end{align}
Rearranging $\eta_1 = \sqrt{\zeta_1^2 + \alpha_1} - \zeta_1$ gives $\alpha_1 = \eta_1^2 + 2 \eta_1 \zeta_1$. Adding $2 \eta_1^2 A_1^\tp P A_1$ to both sides of \eqref{eq:cor_intermediate_lmi} and substituting $\alpha_1 = \eta_1^2 + 2 \eta_1 \zeta_1$ gives exactly \eqref{eq:NLMI_single2}.
Thus the condition \eqref{eq:NLMI} of Thm. \ref{thm:robust_multi_shared_lyap} is satisfied by ${\eta_1 = \sqrt{\zeta_1^2 + \alpha_1} - \zeta_1}$. Applying Thm. \ref{thm:robust_multi_shared_lyap} completes the proof.
\end{proof}

If all robustness bounds $\eta_i$ in Theorem \ref{thm:robust_multi_shared_lyap} are chosen proportional to $\sqrt{\zeta_i^2 + \alpha_i} -\zeta_i$ (like in Cor. \ref{cor:robust_multi_shared_lyap_single}), we obtain the following corollary:

\begin{Cor} \label{cor:robust_multi_shared_lyap_prop}
Suppose the system in \eqref{eq:mult_system} is mean-square stable with $\EE [ \gamma_{ti} ] = 0$, $\EE [ \gamma_{ti}^2 ] = \alpha_i > 0$. 
Fix a $Q \succeq I$ and the solution $P \succ 0$ to \eqref{eq:glyap_p}.
Let $\eta_i > 0$ be scalars which satisfy \eqref{eq:NLMI} and are chosen proportional to $\sqrt{\zeta_i^2 + \alpha_i} -\zeta_i$
where $\zeta_i$ are scalars which marginally satisfy
\begin{align} \label{eq:NLMI_upr}
    \frac{1}{\sqrt{\zeta_i^2 + \alpha_i} - \zeta_i} Q + 2 \zeta_i A_i^\tp P A_i \succeq (A^\tp P A_i + A_i^\tp P A)^+ .
\end{align} 
Then the deterministic system in \eqref{eq:det_system}
is stable for any 
$ 0 \leq \mu_i < \eta_i$
where the $\eta_i$ are upper bounded by
\begin{align*} 
    \eta_i < \sqrt{\zeta_i^2 + \alpha_i} -\zeta_i  < \sqrt{\alpha_i} 
\end{align*}
\end{Cor}

\begin{proof}
    The proof proceeds by contradiction. Suppose
    \begin{align*}
        \eta_i = \sqrt{\zeta_i^2 + \alpha_i} - \zeta_i
    \end{align*}
    From \eqref{eq:NLMI_upr} and using an argument identical to Corollary \ref{cor:robust_multi_shared_lyap_single} we have
    \begin{align*}
        Q + 2 \eta_i \zeta_i A_i^\tp P A_i 
        & \succeq \eta_i \left( A^\tp P A_i + A_i^\tp P A \right)^+      
    \end{align*}    
    Summing over all the noises,
    \begin{align} \label{eq:sum_contradict}
        pQ + \sum_{i=1}^p 2 \eta_i \zeta_i A_i^\tp P A_i 
        & \succeq \sum_{i=1}^p \eta_i \left( A^\tp P A_i + A_i^\tp P A \right)^+      
    \end{align}
    Substituting $\alpha_i = \eta_i^2 + 2 \eta_i \zeta_i$, the matrix inequality in \eqref{eq:NLMI} reduces to
    \begin{align*}
        pQ + \sum_{i=1}^p 2 \eta_i \zeta_i A_i^\tp P A_i 
        & \succeq 
        \sum_{i=1}^p \eta_i \left( A^\tp P A_i + A_i^\tp P A \right)^+
        + \sum_{i=1}^p \sum_{j \neq i} \eta_i \eta_j \left( A_i^\tp P A_j + A_j^\tp P A_i \right)^+ \nonumber 
    \end{align*}
    which is a contradiction; we need the additional terms 
    \begin{align}
        \sum_{i=1}^p \sum_{j \neq i} \eta_i \eta_j \left( A_i^\tp P A_j + A_j^\tp P A_i \right)^+
    \end{align}
    on the right-hand side of \eqref{eq:sum_contradict} in order to match \eqref{eq:NLMI} in Theorem \ref{thm:robust_multi_shared_lyap}, which shows that the bounds $\eta_i$ must be less than $ \sqrt{\zeta_i^2 + \alpha_i} -\zeta_i $.
\end{proof}

Corollaries \ref{cor:robust_multi_shared_lyap_single} and \ref{cor:robust_multi_shared_lyap_prop} go towards showing the functional dependence of upper bounds of the robustness margins on the multiplicative noise variance, namely a $\sqrt{\alpha_i}$ relation. Significantly, the robustness margins collapse to nothing when the variances are all zero and increase monotonically with increasing noise variances.

\subsection{Conservative simplifications}
It can be shown that $\frac{1}{\sqrt{\zeta_i^2 + \alpha_i} - \zeta_i}$ is convex in $\zeta_i$, so any linearization (first-order Taylor series expansion) will be a global underestimator of this function. Thus a conservative solution can be found by linearization, yielding a convex semidefinite constraint which can be expressed as a generalized eigenvalue problem which can be solved efficiently. For example, linearizing $\frac{1}{\sqrt{\zeta_i^2 + \alpha_i} - \zeta_i}$ about $\zeta_i = 0$ yields $\frac{1}{\sqrt{\alpha_i}} + \frac{1}{\alpha_i} \zeta_i $. This is worked out in the following lemma:
\begin{Lem}
Define $A$, $A_i$, $\alpha_i$, $P$, $Q$ as in Cor. \ref{cor:robust_multi_shared_lyap_prop}.
Let $\lambda_i$ be the maximum generalized eigenvalue which solves
\begin{align*}
    \left[ \left( A^\tp P A_i + A_i^\tp P A \right)^+ - \frac{1}{\sqrt{\alpha}} Q \right] v = \lambda_i \left[ \frac{1}{\alpha} Q + 2 A_i^\tp P A_i \right] v .
\end{align*}
Then $\zeta_i \geq \lambda_i$ satisfies \eqref{eq:NLMI_upr}.
\end{Lem}

\begin{proof}
By Lemma \ref{lem:gen_eig} we have the semidefinite bound 
\begin{align*}
    \lambda_i \left( \frac{1}{\alpha_i}  Q + 2 A_i^\tp P A_i \right)
    & \succeq 
    \left( A^\tp P A_i + A_i^\tp P A \right)^+  -  \frac{1}{\sqrt{\alpha_i}} Q .
\end{align*}
Rearranging,
\begin{align*}    
    \left(\frac{1}{\sqrt{\alpha_i}} + \frac{1}{\alpha_i} \lambda_i \right) Q + 2 \lambda_i A_i^\tp P A_i 
    & \succeq 
    \left( A^\tp P A_i + A_i^\tp P A \right)^+  .
\end{align*}
Since $\frac{1}{\sqrt{\lambda_i^2 + \alpha_i} - \lambda_i}$ is a convex function of $\lambda_i$,
\begin{align*}
    \frac{1}{\sqrt{\lambda_i^2 + \alpha_i} - \lambda_i} \geq \frac{1}{\sqrt{\alpha_i}} + \frac{1}{\alpha_i} \lambda_i
\end{align*}
and thus
\begin{align*}
    \frac{1}{\sqrt{\lambda_i^2 + \alpha_i} - \lambda_i} Q + 2 \lambda_i A_i^\tp P A_i
    & \succeq 
    \left( A^\tp P A_i + A_i^\tp P A \right)^+ 
\end{align*}
which is exactly the constraint in \eqref{eq:NLMI_upr} with $\zeta_i = \lambda_i$. Noting that $\frac{1}{\sqrt{\zeta_i^2 + \alpha_i} - \zeta_i}$ is nondecreasing in $\zeta_i$ completes the proof.
\end{proof}

Similarly, an even more conservative bound is obtained by neglecting the contribution of $2 \zeta_i A_i^\tp P A_i$ in \eqref{eq:NLMI_upr}, again resulting in a generalized eigenvalue problem.

\begin{Lem}
Define $A$, $A_i$, $\alpha_i$, $P$, $Q$ as in Cor. \ref{cor:robust_multi_shared_lyap_prop}.
Let $\lambda_i$ be the maximum generalized eigenvalue which solves
\begin{align*}
    \left( A^\tp P A_i + A_i^\tp P A \right)^+ v = \lambda_i Q v .
\end{align*}
Then $\zeta_i \geq \frac{1}{2} \left( \alpha \lambda_i - \frac{1}{\lambda_i} \right)$ satisfies \eqref{eq:NLMI_upr}.
\end{Lem}

\begin{proof}
By Lemma \ref{lem:gen_eig} we have the semidefinite bound 
\begin{align*}
    \lambda_i Q  & \succeq \left( A^\tp P A_i + A_i^\tp P A \right)^+
\end{align*}
Setting
\begin{align*}
    \lambda_i < \frac{1}{\sqrt{\zeta_i^2 + \alpha_i} - \zeta_i}
\end{align*}
and rearranging yields
\begin{align*}
    \zeta_i \geq \frac{1}{2} \left( \alpha \lambda_i - \frac{1}{\lambda_i} \right)
\end{align*}
and
\begin{align*}
    \left( \frac{1}{\sqrt{\zeta_i^2 + \alpha_i} - \zeta_i} \right) Q  & \succeq \left( A^\tp P A_i + A_i^\tp P A \right)^+
\end{align*}
Adding $2 \zeta_i A_i^\tp P A_i \succeq 0$ to the left side gives exactly the constraint in \eqref{eq:NLMI_upr}.
\end{proof}

\section{Robustness via stability of auxiliary systems} \label{sec:robust_aux}

Now, instead of requiring the same Lyapunov function to ensure mean-square stability of a stochastic system and stability of a perturbed deterministic system with the same nominal $A$, we construct auxiliary stochastic systems whose mean-square stability implies deterministic stability of the ``target'' perturbed deterministic system. Such an approach can be fundamentally more flexible than using a shared Lyapunov function since the open-loop dynamics of the auxiliary system are permitted to be significantly less stable.

\clearpage
\begin{Thm} \label{thm:robust_multi_aux}
Suppose the stochastic system
\begin{align*}
    x_{t+1} = \left( A \sqrt{1+\sum_{i=1}^{p} \eta_i} + \sum_{i=1}^{p} \gamma_{ti} A_i \right) x_t
\end{align*}
with $\EE [\gamma_{ti}^2] = \alpha_i \geq \eta_i \left( 1 + \sum_{j=1}^p \eta_j \right)$, $\eta_i \geq 0$ is mean-square stable.
Then the deterministic system
\begin{align*}
    x_{t+1} = \Big(A + \sum_{i=1}^p \mu_i A_i \Big) x_t
\end{align*}
is stable for all $|\mu_i| < \eta_i$.
\end{Thm}

\begin{proof}
Mean-square stability implies $\exists \ P$ such that
\begin{align}
    P   
    & \succ 
    \left(\sqrt{1+\sum_{i=1}^p \eta_i} A \right)^\tp P \left(\sqrt{1+\sum_{i=1}^p \eta_i} A \right)  
    + \sum_{i=1}^p \eta_i \left( 1 + \sum_{j=1}^p \eta_j \right) A_i^\tp P A_i \nonumber \\
    & = \Big( 1+\sum_{i=1}^p \eta_i \Big) A^\tp P A + \sum_{i=1}^p \eta_i \Big( 1 + \sum_{j=1}^p \eta_j \Big) A_i^\tp P A_i \nonumber \\
    & \succ A^\tp P A + \sum_{i=1}^p \eta_i A^\tp P A + \sum_{i=1}^p \eta_i A_i^\tp P A_i  
    + \sum_{i=1}^p  \sum_{j=1}^p \eta_i \eta_j A_i^\tp P A_i \nonumber \\
    & \succeq A^\tp P A + \sum_{i=1}^p \eta_i (A^\tp P A_i + A_i^\tp P A)  
    + \sum_{i=1}^p  \sum_{j=1}^p \eta_i \eta_j A_i^\tp P A_i \label{eq:succ_sqr} \\
    & = A^\tp P A + \sum_{i=1}^p \eta_i (A^\tp P A_i + A_i^\tp P A)
    + \sum_{i=1}^p \eta_i^2 A_i^\tp P A_i + \sum_{i=1}^p \sum_{j \neq i} \eta_i \eta_j A_i^\tp P A_i \nonumber \\
    & \succeq A^\tp P A + \sum_{i=1}^p \eta_i (A^\tp P A_i + A_i^\tp P A)  
    + \sum_{i=1}^p \eta_i^2 A_i^\tp P A_i + \sum_{i=1}^p \sum_{j \neq i} \eta_i \eta_j A_i^\tp P A_j \nonumber \\
    & = \Big(A + \sum_{i=1}^p \eta_i A_i\Big)^\tp P \Big(A + \sum_{i=1}^p \eta_i A_i\Big) \nonumber
\end{align}
By symmetry of the terms $\sum_{i=1}^p \eta_i A^\tp P A + \sum_{i=1}^p \eta_i A_i^\tp P A_i$, the same argument can be applied for each sign combination of $\eta_i$ i.e. $\pm \eta_1, \pm \eta_2, \ldots, \pm \eta_p$ from \eqref{eq:succ_sqr} onward, which together prove stability of $A + \sum_{i=1}^p k_i \eta_i A_i$ for any $k_i \in \{-1,+1\}$ with the same Lyapunov matrix $P$. By an argument from Schur complements (see e.g. \cite{Corless1994,boyd1994linear}), this is necessary and sufficient for any convex combination of $A + \sum_{i=1}^p k_i \eta_i A_i$ to be also stable using $P$, completing the proof.
\end{proof}

\begin{Rem}
    The condition $\EE [\gamma_{ti}^2] = \alpha_i \geq \eta_i \left( 1 + \sum_{j=1}^p \eta_j \right)$ places an upper bound on the robustness margins $\eta_i$ which is related to the multiplicative noise variances $\alpha_i$. In the case of $p = 1$, this reduces to 
    $
        \eta_1 < \frac{1}{2} \left( \sqrt{1+4 \alpha_1} - 1 \right) .
    $
\end{Rem}

At first glance the condition of Thm. \ref{thm:robust_multi_shared_lyap} may seem overly restrictive since it requires mean-square stability with a \emph{scaled} $A$ matrix; indeed such a procedure is somewhat limiting in the open-loop setting since this can make the plant unstable. However, in the control design setting this \emph{essentially does not matter} since the gain can be made larger to compensate, and because a simple scaling of $A$ does not affect controllability of the pair $(A, B)$; to see this, simply note that the rank of the controllability matrix 
$\begin{bmatrix}
B & AB & \ldots & A^{n-1} B
\end{bmatrix}
$ is unaffected by a nonzero scaling of $A$. The work of \cite{bernstein1987} similarly leverages this fact.

\clearpage
\section{Input uncertainties and robust control design algorithms} \label{sec:control_design}

In the case where there are uncertainties in the input matrix $B$ under closed-loop state feedback, Theorems \ref{thm:robust_multi_shared_lyap} and \ref{thm:robust_multi_aux} are easily modified by simply substituting
\begin{align*}
    A                \leftarrow A+BK , \quad
    \{ A_i \}        \leftarrow \{ A_i \} \cup \{ B_j K\} , \quad
    \{ \alpha_i \}   \leftarrow \{\alpha_i\} \cup \{\beta_j \} , \quad
    p                \leftarrow p + q ,
\end{align*}
yielding a set of $p+q$ robustness bounds $\{ \eta^\prime_i \} = \{\eta_i\} \cup \{ \psi_j\}$
which ensure stability of 
\begin{align} \label{eq:closed_loop_det_sys}
    x_{t+1} = \Big(A+BK + \sum_{i=1}^p \mu_i A_i + \sum_{j=1}^q \nu_j B_j K \Big) x_t 
\end{align}
where $0 \leq \mu_i < \eta_i$, $0 \leq \nu_j < \psi_j$ (bounds in negative directions also assured for Thm. \ref{thm:robust_multi_aux}).
These results are formulated as Algorithms \ref{alg:robust_control1} and \ref{alg:robust_control2} for generating optimal, maximally robust controllers.
Note that Algorithm \ref{alg:robust_control1} gives unidirectional bounds while Algorithm \ref{alg:robust_control2} gives bidirectional bounds; it is useful to retain the unidirectional bounds of Algorithm \ref{alg:robust_control1} in order to realize the potentially larger robustness margins in opposing directions.

\begin{algorithm2e}[!ht]
\caption{Robust control design} \label{alg:robust_control1}
\DontPrintSemicolon
\KwIn{Controllable nominal pair $(A, B)$, cost matrices $Q \succ 0$, $R \succ 0$, uncertainty directions $A_i$, $B_j$ and magnitudes $\theta_i > 0$, $\phi_j > 0$.}
\KwOut{Gain matrix $K$ and margins $\eta_i$, $\psi_j$ such that \eqref{eq:closed_loop_det_sys} is stable for all $0 \leq \mu_i < \eta_i$, $0 \leq \nu_j < \psi_j$.}
\Define scalar $z$ and scaled multiplicative noise variances 
${\alpha_i = \theta_i \times z}$, and
${\beta_j =  \phi_j \times z}$ \;
\Find the largest $z^*$ which still admits a solution to 
$P = \gare(A,B,Q,R,\alpha_i,\beta_j,A_i,B_j)$ via bisection \;
\Define scalar $y$ and scaled uncertainty magnitudes $\eta_i = \theta_i \times y$, $\psi_j = \phi_j \times y$\;
\Find the largest scaling $y^*$ via bisection which satisfies
\begin{align*}
& Q+K^\tp R K + \sum_{i=1}^{p+q} {\alpha_i}^\prime {A_i^\prime}^\tp P A_i^\prime \\
& \succeq 
\sum_{i=1}^{p+q} \eta_i^\prime \left( {A_i^\prime}^\tp P (A+BK) + (A+BK)^\tp P A_i^\prime \right)^+ \\
& \quad + \sum_{i=1}^{p+q} \sum_{j=1}^{p+q} \eta_i^\prime \eta_j^\prime \left( {A_i^\prime}^\tp P A_j^\prime + {A_j^\prime}^\tp P A_i^\prime \right)^+ \nonumber
\end{align*}
where
$\{ A_i^\prime \} = \{ A_i \} \cup \{ B_j K\}$,
$\{ \alpha_i^\prime \}  = \{\alpha_i\} \cup \{\beta_j \}$,
and
$\{ \eta_i^\prime \} = \{ \eta_i \} \cup \{ \psi_j \}$ \;
\Return control law
${K = -  \left(R + B^\tp P B + z^* \sum_{j=1}^q   \phi_j  B_j^\tp P B_j \right)^{-1} B^\tp P A}$ \\
and margins $\eta_i = \theta_i \times y^*$, $\psi_j = \phi_j \times y^*$
\;
\end{algorithm2e}

\begin{algorithm2e}[!ht] 
\caption{Robust control design} \label{alg:robust_control2}
\DontPrintSemicolon
\KwIn{Controllable nominal pair $(A, B)$ , cost matrices $Q \succ 0$, $R \succ 0$, uncertainty directions $A_i$, $B_j$ and magnitudes $\theta_i > 0$, $\phi_j > 0$.}
\KwOut{Gain matrix $K$ and robustness margins $\eta_i$, $\psi_j$ such that \eqref{eq:closed_loop_det_sys} is stable for all $|\mu_i| < \eta_i$, $|\nu_j| < \psi_j$.}
\Define scalar $y$ and scaled uncertainty magnitudes 
$\eta_i = \theta_i \times y$, $\psi_j = \phi_j \times y$ \;
\Define scaled multiplicative noise variances 
${\alpha_i = \eta_i \left( 1 + \sum_{j=1}^p \eta_j + \sum_{k=1}^q \psi_k \right)}$, and 
${\beta_j = \psi_i \left( 1 + \sum_{i=1}^p \eta_i + \sum_{k=1}^q \psi_k \right)}$ \;
\Define scalar $z(y) = \sqrt{1+\sum_{i=1}^{p} \eta_i+\sum_{j=1}^{q} \psi_i}$, and scaled system matrices $A_z = A \times z$, $B_z = B \times z$ \;
\Find the largest $y^*$ which still admits a solution to 
$P = \gare(A_z,B_z,Q,R,\alpha_i,\beta_j,A_i,B_j)$ via bisection \;
\Return control law
${K = -  \left(R + B_{z}^\tp P B_{z} + \sum_{j=1}^q \beta_j B_j^\tp P B_j \right)^{-1} B_{z}^\tp P A_{z}}$ 
where quantities $P$, $\beta_j$ and $z$ are evaluated at $y^*$, and
margins $\eta_i = \theta_i \times y^*$, $\psi_j = \phi_j \times y^*$
\;
\end{algorithm2e}

\section{Numerical results} \label{sec:numerical_results}

Here we consider an inverted pendulum with a torque-producing actuator whose dynamics have been linearized about the vertical equilibrium. In continuous-time the dynamics are
\begin{align*}
    \dot{x} = 
    \underbrace{
    \begin{bmatrix}
    0 & 1 \\
    m_c & 0
    \end{bmatrix}
    }_{A_c}
    x
    +
    \underbrace{
    \begin{bmatrix}
    0 \\
    1
    \end{bmatrix}  
    }_{B_c}
    u
\end{align*}
where $m_c$ is a normalized mass constant. A forward Euler discretization with step size $\Delta t$ yields
\begin{align*}
    A  = I + A_c\Delta t = 
    \begin{bmatrix}
    1 & \Delta t \\
    m_c \Delta t & 1 
    \end{bmatrix}, \quad
    B  = B_c \Delta t = 
    \begin{bmatrix}
    0 \\
    \Delta t
    \end{bmatrix}     
\end{align*}
Uncertainty on the mass constant $m_c$ corresponds to uncertainty on the $(2,1)$ entry of $A$.
We consider an example where the true mass constant is $\bar{m}_c = 10$, but the nominal model underestimates it as $m_c = 5$; such a situation could easily arise during the initial phase of system identification in adaptive control with noisy measurements, or in time-varying scenarios such as a robot arm picking up a heavy load. We take a step size $\Delta t = 0.1$.
The problem data is then
\begin{align*}
    \bar{A} = 
    \begin{bmatrix}
    1 & 0.1 \\
    1 & 1
    \end{bmatrix}, \
    A  = 
    \begin{bmatrix}
    1 & 0.1 \\
    0.5 & 1    
    \end{bmatrix}, \
    \bar{B} = B = 
    \begin{bmatrix}
    0 \\
    0.1
    \end{bmatrix} , \
    Q = R  =
    \begin{bmatrix}
    1 & 0 \\
    0 & 1
    \end{bmatrix}, \
    A_1 = 
    \begin{bmatrix}
    0 & 0 \\
    1 & 0
    \end{bmatrix}, \
    \theta_1 = 1
\end{align*}

Applying Algorithms \ref{alg:robust_control1}, \ref{alg:robust_control2}, and certainty-equivalent control design, we obtained the results in Table \ref{tab:table}.
We found the sets of true $\bar{A}$ matrices stabilized by the controls from Algos. \ref{alg:robust_control1} and \ref{alg:robust_control2} were
${\bar{A} \in 
\begin{bmatrix}
1 & 0.1 \\
0.1 + \mu_1 & 1
\end{bmatrix}}$ where $|\mu_1| < 3.970$ and $|\mu_1| < 6.997$ respectively. Stability of all systems within these sets was empirically verified by a fine grid search using 10000 samples of $\mu_1$ in each interval. 
Both robustness sets happened to include the true matrix $\bar{A}$, so the robust controls were guaranteed to stabilize the true system, confirmed by  $\rho(\bar{A}+\bar{B} K) < 1$.
By contrast, the certainty-equivalent control failed to stabilize the true system.
This can be understood intuitively; the pendulum had a larger mass in reality than in the nominal model, so a larger control effort was necessary to stabilize the pendulum and prevent it from falling over. 
Although on this particular example Algorithm 1 gave a larger (unidirectional) robustness margin, in general this not need hold; certain problem instances admit much larger robustness margins using Algorithm 2 relative to Algorithm 1. Thus, our two algorithms may be considered complementary from a control design standpoint.

Code which implements this example is available at: \\
\url{https://github.com/TSummersLab/robust-control-multinoise}.

\begin{table*}[!ht] 
\caption{Stability results for robust control of an inverted pendulum.} \label{tab:table}
\centering
\normalsize
\begin{tabular}{ c c c c c }
\toprule
Parameter & Open-loop & Certainty-equivalent & Algorithm \ref{alg:robust_control1} & Algorithm \ref{alg:robust_control2}  \\
\midrule\\
\addlinespace[-3ex]
$K$ 
&
$\begin{bmatrix}
0 \! & \! 0
\end{bmatrix}$
&
$\begin{bmatrix}
-9.14 \! & \! -4.15
\end{bmatrix}$
&
$\begin{bmatrix}
-103.87 \! & \! -19.85
\end{bmatrix}$
&
$\begin{bmatrix}
-104.52 \! & \! -19.94
\end{bmatrix}$ 
\\
$\rho(\bar{A}+\bar{B} K)$
& 1.316
& 1.019
& 0.222
& 0.225 
\\
$\rho(A+BK)$
& 1.223
& 0.833
& 0.060
& 0.020
\\
$\eta_1$
& -
& -
& 6.997
& 3.970
\\
$\underset{0 \leq \mu_1 < \eta_1}{\max}{\rho(A+BK+\mu_1 A_1)}$
& -
& -
& 0.841
& 0.632
\\
\bottomrule
\end{tabular}
\end{table*}

\clearpage
\section{Conclusion and Future Work} \label{sec:conclusion}

This work gives an effective methodology for certifying robustness and designing robust controllers with favorable properties and flexibility relative to competing approaches.

Direct extensions to this work include finding sharper bounds, e.g., via alternate auxiliary systems analogous to the one used in Section \ref{sec:robust_aux}, and handling nonlinear dependence of the dynamics and/or noise on states and inputs. Future work will integrate the results of this work with adaptive model-based learning control for an end-to-end control framework which gracefully transitions from maximal robustness to maximal performance according to empirical uncertainties.

\bibliographystyle{plain}
\bibliography{bibliography.bib}

\end{document}

%% file: style_PM.tex
\makeatletter
\def\@settitle{\begin{center}%
		\baselineskip14\p@\relax
		\normalfont\LARGE\scshape\bfseries
		\@title
	\end{center}%
}
\makeatother

\makeatletter

\def\subsection{\@startsection{subsection}{2}%
	\z@{.5\linespacing\@plus.7\linespacing}{.5\linespacing}%
	{\normalfont\bfseries}}

\def\subsubsection{\@startsection{subsubsection}{3}%
	\z@{.5\linespacing\@plus.7\linespacing}{.5\linespacing}%
	{\normalfont\itshape}}

\usepackage[usenames, dvipsnames]{xcolor}

\definecolor{darkblue}{rgb}{0.0, 0.0, 0.45}

\usepackage[colorlinks	= true,
			raiselinks	= true,
			linkcolor	= darkblue, 
			citecolor	= Mahogany,
			urlcolor	= ForestGreen,
			pdfauthor	= {},
			pdftitle	= {},
			pdfkeywords	= {},
			pdfsubject	= {},
			plainpages	= false]{hyperref}

\usepackage{dsfont,amssymb,amsmath,graphicx} 
\usepackage{amsfonts,dsfont,mathtools,mathrsfs,amsthm} 
\mathtoolsset{showonlyrefs=true}
\usepackage[amssymb, thickqspace]{SIunits}
\usepackage{fancyhdr,setspace}

\usepackage{nicefrac}       
\usepackage{microtype}      
\usepackage{tabto}

\usepackage{enumerate,enumitem}
\usepackage{subfig}
\usepackage{booktabs}

\usepackage{bbm}

\allowdisplaybreaks
\date{\today}

%% file: commands_PM.tex
\theoremstyle{theorem}
\newtheorem{Thm}{Theorem}[section]

\newtheorem{Lem}[Thm]{Lemma}
\newtheorem{Cor}[Thm]{Corollary}

\newtheorem{Def}[Thm]{Definition}

\newtheorem{Rem}[Thm]{Remark}

\theoremstyle{remark}

\newcommand{\EE}{\mathds{E}}
\newcommand{\RR}{\mathbb{R}}
\newcommand{\tp}{\intercal}

\DeclareMathOperator{\gare}{gare}